\documentclass[a4paper,12pt,twoside,leqno,final]{amsart}
\usepackage{amsmath}
\usepackage{amssymb}

\newtheorem{thm}{Theorem}[section]
\newtheorem{lem}[thm]{Lemma}

\theoremstyle{definition}

\newcommand{\C}{{\mathbb C}}
\newcommand{\D}{{\mathbb D}}

\newcommand{\T}{{\mathbb T}}
\newcommand{\Z}{{\mathbb Z}}
\newcommand{\N}{{\mathbb N}}

\newcommand{\La}{\Lambda}

\newcommand{\eps}{\varepsilon}

\newcommand{\f}{\frac}
\newcommand{\ov}{\overline}

\newcommand{\de}{\delta}
\newcommand{\la}{\lambda}
\newcommand{\ze}{\zeta}
\renewcommand{\th}{\theta}

\newcommand{\ph}{\varphi}

\numberwithin{equation}{section}

\title[Extreme points in quotients of Hardy spaces]
{Extreme points in quotients of Hardy spaces}
\author{Konstantin M. Dyakonov}
\address{Departament de Matem\`atiques i Inform\`atica, IMUB, BGSMath, Universitat de Barcelona, Gran Via 585, E-08007 Barcelona, Spain}
\address{ICREA, Pg. Llu\'is Companys 23, E-08010 Barcelona, Spain}
\email{konstantin.dyakonov@icrea.cat}
\thanks{Supported in part by grants PID2021-123405NB-I00 and PID2024-160033NB-I00 from El Ministerio de Ciencia, Innovaci\'on y Universidades (Spain).}

\begin{document}
\begin{abstract}
In the Hardy spaces $H^1$ and $H^\infty$, there are neat and well-known characterizations of the extreme points of the unit ball. We obtain counterparts of these classical theorems when $H^1$ (resp., $H^\infty$) gets replaced by the quotient space $H^1/E$ (resp., $H^\infty/E$), under certain assumptions on the subspace $E$. In the $H^1$ setting, we also treat the case where the underlying space is taken to be the kernel of a Toeplitz operator.
\end{abstract}

\maketitle

\section{Introduction and statement of results}

Let $\T$ stand for the circle $\{\ze\in\C:|\ze|=1\}$ and $m$ for its Haar measure. For $1\le p\le\infty$, the Lebesgue space $L^p=L^p(\T,m)$ (of complex-valued functions) is defined in the usual manner and equipped with the natural norm $\|\cdot\|_p$. With a function $f\in L^1$ we associate the sequence of its {\it Fourier coefficients} 
$$\widehat f(k):=\int_\T\ov\ze^kf(\ze)\,dm(\ze),\qquad k\in\Z,$$
and the set 
$$\text{\rm spec}\,f:=\{k\in\Z:\,\widehat f(k)\ne0\},$$
known as the {\it spectrum} of $f$. The {\it Hardy space} $H^p$ is then defined, for $p$ as above, by
$$H^p:=\{f\in L^p:\,\text{\rm spec}\,f\subset\mathbb Z_+\},$$
where $\mathbb Z_+:=\{0,1,2,\dots\}$. Finally, we write $H^p_0:=\{f\in H^p:\,\widehat f(0)=0\}$. 

\par The Poisson integral (i.e., harmonic extension) of an $H^p$ function is holomorphic on the unit disk 
$$\D:=\{z\in\C:|z|<1\},$$
and this property actually characterizes the $H^p$ functions among those in $L^p$. With this in mind, we may (and often will) identify a given $H^p$ function with its Poisson extension into $\D$ and view it as a holomorphic function therein. See any of \cite{G, Hof, K} for the underlying theory and basic facts about Hardy spaces.

\par Now let $X$ be a Banach space and let
$$\text{\rm ball}\,(X):=\{x\in X:\,\|x\|\le1\}$$
be its closed unit ball. As usual, an element $x$ of $\text{\rm ball}\,(X)$ is said to be an {\it extreme point} of the ball if it is not expressible in the form $x=\f12(u+v)$ with two distinct points $u,v\in\text{\rm ball}\,(X)$. Equivalently, $x$ is an extreme point of $\text{\rm ball}\,(X)$ if and only if the only vector $y\in X$ satisfying $\|x+y\|\le1$ and $\|x-y\|\le1$ is $y=0$. 

\par A remarkable theorem of de Leeuw and Rudin (see \cite{dLR} or \cite[Chapter 9]{Hof}) states that a unit-norm function $f\in H^1$ is an extreme point of $\text{\rm ball}\,(H^1)$ if and only if $f$ is {\it outer} (i.e., $f$ has no zeros in $\D$ and the harmonic function $\log|f(z)|$ is the Poisson integral of its boundary values). Another classical result asserts that a unit-norm function $h\in H^\infty$ is an extreme point of $\text{\rm ball}\,(H^\infty)$ if and only if 
$$\int_\T\log(1-|h|)\,dm=-\infty$$
(or equivalently, $\log(1-|h|)\notin L^1$); see \cite[Section V]{dLR} or \cite[Chapter 9]{Hof}. Our purpose here is to adapt, or rather transfer, these characterizations---along with a more recent result from \cite{DPAMS} that generalizes the former theorem---to the appropriate quotient spaces. In doing so, we are influenced by an approach suggested by Koosis in \cite{KDuke} to describe the extreme points of $\text{\rm ball}\,(L^\infty/H^\infty)$; specifically, we refer to \cite[Theorem 4.1]{KDuke} and the proof that follows it. Meanwhile, we pause to recall a bit of terminology and introduce another piece of notation.

\par Given a Banach space $X$ and a closed subspace $Y\subset X$, the quotient space $X/Y$ consists of the cosets 
$$x+Y:=\{x+y:\,y\in Y\}$$
with $x\in X$, the norm being defined by 
$$\|x+Y\|_{X/Y}:=\inf\{\|x+y\|_X:\,y\in Y\}.$$
With each $x\in X$ we associate the (possibly empty) set
$$\mathcal M(x;X,Y):=\left\{\xi\in x+Y:\,\|\xi\|_X=\|x+Y\|_{X/Y}\right\},$$
whose elements are precisely the points of minimum norm in $x+Y$. It should be noted that $\mathcal M(x;X,Y)$ is a convex set. We also observe that $\xi\in\mathcal M(x;X,Y)$ if and only if $x-\xi$ is a best approximation for $x$ in $Y$.

\par Now, if 
$$\mathcal M(x;X,Y)\ne\emptyset\quad\text{\rm for every }x\in X$$
(or equivalently, if every $x\in X$ has a best approximation in $Y$), then $Y$ is said to be {\it proximinal} in $X$. It is a classical result that if $\dim Y<\infty$, then $Y$ is proximinal; see, e.g., \cite[p.\,132]{Deu}. Another useful fact, based on the Banach--Alaoglu theorem, is that if $X$ is a dual space and $Y$ is weak-star closed in $X$, then $Y$ is proximinal in $X$; see \cite[p.\,239]{Phe} for details. On the other hand, it is known that every nonreflexive Banach space has a nonproximinal closed subspace; cf. \cite[p.\,140]{Deu} or \cite[p.\,253]{Phe}.

\par Our first result characterizes the extreme points of $\text{\rm ball}\,(H^1/E)$ for certain subspaces $E$ of $H^1$. To describe the subspaces $E$ to be dealt with, we first introduce the notation $H^1_\Phi$ for the annihilator in $H^1$ of a given set $\Phi\subset C(\T)$. In other words, once such a set $\Phi$ is fixed, we write 
$$H^1_\Phi:=\left\{f\in H^1:\,\int_{\T}f\phi\,dm=0\,\,\text{\rm for all }\phi\in\Phi\right\}.$$
As special cases, we mention the subspaces 
$$H^1(\La):=\{f\in H^1:\,\text{\rm spec}\,f\subset\La\}$$
corresponding to subsets $\La$ of $\mathbb Z_+$. Now, the subspaces $E$ that appear in Theorem \ref{thm:hone} below are of the form 
$$\th H^1_\Phi:=\left\{\th f:\,f\in H^1_\Phi\right\},$$
where $\th$ is an {\it inner function} (i.e., $\th\in H^\infty$ and $|\th|=1$ almost everywhere on $\T$). Note that if $\Phi=\{0\}$, or if $\Phi$ is empty, then $H^1_\Phi$ is just $H^1$ and the above formula produces the shift-invariant subspaces $\th H^1$. 

\begin{thm}\label{thm:hone} Let $E=\th H^1_\Phi$, with $\th$ an inner function and $\Phi$ a subset of $C(\T)$. If $\th$ is constant, assume in addition that $\Phi$ contains the constant function $1$. Given a unit-norm coset $f+E$ in $H^1/E$, the following are equivalent: 
\par{\rm(i.1)} $f+E$ is an extreme point of $\text{\rm ball}\,(H^1/E)$.
\par{\rm(ii.1)} There is an outer function $F\in H^1$ such that 
$$\mathcal M\left(f;H^1,E\right)=\{F\}.$$
\end{thm}

\par Obviously, condition (ii.1) implies in particular that the set $\mathcal M\left(f;H^1,E\right)$ contains exactly one element, and the reader may wonder whether this is automatic. While all of our subspaces $E$ are proximinal in $H^1$ (see Lemma \ref{lem:lemone} below), for some of them we actually have 
$$\#\mathcal M\left(f;H^1,E\right)=1\quad\text{\rm whenever }f\in H^1,$$
meaning that every such $f$ has a unique best approximation in $E$. This uniqueness property holds, e.g., when $E=\th H^1$. We also have uniqueness when $E$ is $\th H^1(\La)$ or just $H^1(\La)$, provided that $\La$ is an infinite arithmetic progression in $\N:=\{1,2,\dots\}$ with an odd common difference, whereas no such thing is true for generic sets $\La\subset\N$; see \cite[Section 3]{Kah} for details. 

\par To see an example of an extreme coset, take a nonconstant inner function $\th$ and put $c_\th:=(1-|\th(0)|^2)^{-1}$. It is known (cf. \cite{New}) that the coset $c_\th+\th H^1$ has norm $1$ in $H^1/\th H^1$, and moreover, the unique function of minimum norm in that coset is 
$$F_\th:=c_\th\left(1-\ov{\th(0)}\th\right)^2.$$
Since $F_\th$ is outer, Theorem \ref{thm:hone} tells us that $c_\th+\th H^1$ is an extreme point of $\text{\rm ball}\,(H^1/\th H^1)$.

\par We now supplement Theorem \ref{thm:hone} with a related result, where the underlying space $H^1$ gets replaced by a certain subspace thereof, namely by the kernel of a Toeplitz operator. Precisely speaking, we fix a function $\ph\in L^\infty$ and define the {\it Toeplitz operator} $T_\ph$ by putting, for $f\in H^1$, 
$$(T_\ph f)(z):=\int_\T\f{\ph(\ze)f(\ze)}{1-z\ov\ze}\,dm(\ze),\qquad z\in\D.$$
(Note that the image $T_\ph(H^1)$ need not be contained in $H^1$, unless special assumptions are imposed on $\ph$.) We then write $K_1(\ph)$ for the kernel of $T_\ph$ in $H^1$; equivalently, 
$$K_1(\ph):=\{f\in H^1:\,\ov{z\ph f}\in H^1\}.$$
When $K_1(\ph)\ne\{0\}$, it makes sense to look at the extreme points of its unit ball (with respect to the natural norm $\|\cdot\|_1$), and these were determined by the author in \cite[Theorem 6]{DPAMS}. Namely, it was shown there that a function $f\in K_1(\ph)$ with $\|f\|_1=1$ is an extreme point of $\text{\rm ball}\,(K_1(\ph))$ if and only if 
\begin{equation}\label{eqn:innrelprime}
\text{\rm the inner factors of $f$ and $\ov{z\ph f}$ are relatively prime,}
\end{equation}
meaning that there is no nonconstant inner function $J$ such that $f/J\in H^1$ and $\ov{z\ph f}/J\in H^1$. When $\ph\equiv0$, we have $K_1(\ph)=H^1$ and this characterization reduces to de Leeuw and Rudin's theorem on the extreme points of $\text{\rm ball}\,(H^1)$. For further generalizations of their theorem to subspaces of $H^1$, see \cite{DAdv2021, DAdv2022}. 

\par One special case of interest arises when $\ph=\ov\th$, with $\th$ a nonconstant inner function. In this case, $K_1(\ph)$ becomes the so-called {\it model subspace} $H^1\cap\th\ov{H^1_0}$ (and is sure to contain non-null functions). Such subspaces, with $\th$ inner, are in fact precisely the (closed) invariant subspaces of the backward shift operator 
$$f\mapsto\f{f-f(0)}z,\qquad f\in H^1.$$
In the context of model subspaces, condition \eqref{eqn:innrelprime} appeared for the first time in \cite{DKal}.

\par Our next result deals with certain quotient spaces of $K_1(\ph)$. 

\begin{thm}\label{thm:kertoe} Let $\ph\in L^\infty$ be such that $K_1(\ph)\ne\{0\}$, and let $\psi$ be a nonconstant inner function for which the subspace $E_{\ph,\psi}:=K_1(\ph)\cap\psi H^1$ is proximinal in $K_1(\ph)$. Given a unit-norm coset $f+E_{\ph,\psi}$ in $K_1(\ph)/E_{\ph,\psi}$, the following are equivalent: 
\par{\rm(i.2)} $f+E_{\ph,\psi}$ is an extreme point of 
$\text{\rm ball}\,\left(K_1(\ph)/E_{\ph,\psi}\right)$.
\par{\rm(ii.2)} There is a function $h\in K_1(\ph)$ satisfying
$$\mathcal M\left(f;K_1(\ph),E_{\ph,\psi}\right)=\{h\}$$
and such that the inner factors of $h$ and $\ov{z\ph h}$ are relatively prime.
\end{thm}

\par With regard to the hypothesis that $E_{\ph,\psi}$ is proximinal in $K_1(\ph)$, one might argue that this assumption is not explicit enough to be useful (or usable). However, we can easily come up with examples where $\dim E_{\ph,\psi}<\infty$, and since finite-dimensional subspaces are proximinal, Theorem \ref{thm:kertoe} is certainly applicable in these cases. To produce such an example, let $\psi$ be a nonconstant inner function and $B$ a finite Blaschke product; then put $\ph=\ov\psi\ov B$. Given a function $f\in\psi H^1$, we write it as $f=F\psi$ with $F\in H^1$ and note that $f\in K_1(\ph)$ if and only if $\ov z\ov FB\in H^1$. This last condition means that $F\in K_1(\ov B)$, and so 
$$E_{\ph,\psi}=\psi K_1(\ov B).$$
Because the model subspace $K_1(\ov B)=H^1\cap B\ov{H^1_0}$ is finite-dimensional (indeed, its elements are rational functions whose poles are contained among those of $B$), so is $E_{\ph,\psi}$.

\par We finally turn to quotients of $H^\infty$.

\begin{thm}\label{thm:hinfty} Let $E$ be a proper weak-star closed subspace of $H^\infty$. Given a unit-norm coset $f+E$ in $H^\infty/E$, the following are equivalent: 
\par{\rm(i.3)} $f+E$ is an extreme point of $\text{\rm ball}\,(H^\infty/E)$.
\par{\rm(ii.3)} There is a unit-norm function $h\in H^\infty$ with $\log(1-|h|)\notin L^1$ such that 
$$\mathcal M\left(f;H^\infty,E\right)=\{h\}.$$
\end{thm}

\par To keep on the safe side, we recall that $H^\infty$ is the dual of $L^1/H^1_0$ (see \cite[Chapter IV]{G}), so the term \lq\lq weak-star closed" in the above theorem should be understood accordingly. As examples of proper weak-star closed subspaces in $H^\infty$, we mention the shift-invariant ideals $\th H^\infty$ and the model subspaces $H^\infty\cap\th\ov{H^\infty_0}$, each of these associated with a nonconstant inner function $\th$. Further examples include subspaces of the form 
$$H^\infty(\La):=\{f\in H^\infty:\,\text{\rm spec}\,f\subset\La\},$$
with $\La$ a proper subset of $\Z_+$, and more generally, annihilators of nontrivial sets in $L^1/H^1_0$. 

\par We go on to remark that Theorem \ref{thm:hinfty} can be applied in certain situations that are of interest in the context of Nevanlinna--Pick interpolation problems. Let $I\subset\T$ be an open arc with $m(I)<1$ that contains the point $1$, and let $B$ be a Blaschke product whose zeros $\{z_j\}$ satisfy $z_j\to1$. Suppose also that $f$ is an invertible (hence outer) function in $H^\infty$ such that $\|f\|_\infty=1$ and $|f|=1$ on $I$, but $|f|\not\equiv1$ on $\T$. Under these assumptions, $f$ is known to be the only unit-norm $H^\infty$ function that interpolates the values $f(z_j)$ at $z_j$. (See \cite[p.\,145]{G} for a detailed discussion of this example.) Thus, the coset $f+BH^\infty$ has norm $1$ in $H^\infty/BH^\infty$ and 
$$\mathcal M\left(f;H^\infty,BH^\infty\right)=\{f\}.$$
Since $\log(1-|f|)\notin L^1$, Theorem \ref{thm:hinfty} tells us that $f+BH^\infty$ is an extreme point of $\text{\rm ball}\,(H^\infty/BH^\infty)$. At the same time, by \cite[Theorem 4.1]{KDuke}, the coset $f\ov B+H^\infty$ is nonextreme in $\text{\rm ball}\,(L^\infty/H^\infty)$; indeed, the unique unit-norm function in that coset is $f\ov B$ and its modulus is nonconstant on $\T$.

\par Finally, we mention an open question that comes to mind in light of Theorem \ref{thm:hinfty}. Namely, we wonder whether that theorem could be adjusted to yield a description of the {\it strongly extreme} points of $\text{\rm ball}\,(H^\infty/E)$, with $E$ as above. (A unit-norm vector $x$ in a Banach space $X$ is said to be a strongly extreme point of $\text{\rm ball}\,(X)$ if for every $\eps>0$ there is a $\de>0$ such that the inequalities $\|x\pm y\|<1+\de$ imply, for $y\in X$, that $\|y\|<\eps$.) In $H^\infty$, the strongly extreme points of the unit ball are precisely the inner functions; see \cite{CT} for a proof and \cite{DAFM} for a further generalization. Now, for a quotient space $H^\infty/E$, one is tempted to conjecture that a coset $f+E$ of norm $1$ will be a strongly extreme point of the unit ball if and only if the set $\mathcal M\left(f;H^\infty,E\right)$ contains exactly one function, which is inner.

\par We now turn to the proofs of our current results. In the next section, we establish two auxiliary facts to lean upon. In the remaining sections, we use them to prove Theorems \ref{thm:hone}, \ref{thm:kertoe}, and \ref{thm:hinfty}.

\section{Preliminaries}

Two lemmas will be proved. The first of these is needed to make sure that the subspaces $E$ appearing in Theorem \ref{thm:hone} are proximinal. The other one is based on (and provides an abstract version of) an observation due to Koosis, which was key to his proof of \cite[Theorem 4.1]{KDuke}.

\begin{lem}\label{lem:lemone} Let $E=\th H^1_\Phi$, where $\th$ and $\Phi$ satisfy the hypotheses of Theorem \ref{thm:hone}. Then $E$ is proximinal in $H^1$.
\end{lem}

\begin{proof} Given $f\in H^1$, we want to find a function $h_0\in H^1_\Phi$ such that
\begin{equation}\label{eqn:distfthetah}
\inf\left\{\|f+\th h\|_1:\,h\in H^1_\Phi\right\}=\|f+\th h_0\|_1.
\end{equation}
Identifying each function $w\in L^1$ with the absolutely continuous measure $w\,dm$, we may embed $L^1$ (and, in particular, $H^1_\Phi$) isometrically in $M(\T)$, the space of complex Borel measures on $\T$ with the usual total variation norm. Viewing the latter space as the dual of $C(\T)$, under the natural pairing, we note that $H^1_\Phi$ is the annihilator (in $M(\T)$) of the family 
$$\left\{z^n:\,n=1,2,\dots\right\}\cup\Phi$$
contained in $C(\T)$. Therefore, $H^1_\Phi$ is a weak-star closed subspace of $M(\T)$, which in turn implies that $H^1_\Phi$ is proximinal therein (we refer again to \cite[p.\,239]{Phe}).

\par Thus, whenever $f\in H^1$, the integrable function $f\ov\th$ (or, equivalently, the measure $f\ov\th\,dm$) has a best approximation in $H^1_\Phi$. In other words, there exists a function $h_0\in H^1_\Phi$ such that 
$$\inf\left\{\|f\ov\th+h\|_1:\,h\in H^1_\Phi\right\}=\|f\ov\th+h_0\|_1.$$
This last identity obviously coincides with \eqref{eqn:distfthetah}, and we are done.
\end{proof}

\begin{lem}\label{lem:lemtwo} Let $X$ be a Banach space and $Y$ a proximinal subspace of $X$. Suppose further that $x\in X$ and
\begin{equation}\label{eqn:meqextr}
\mathcal M(x;X,Y)=\{x_0\},
\end{equation}
where $x_0$ is an extreme point of $\text{\rm ball}\,(X)$. Then $x+Y$ is an extreme point of $\text{\rm ball}\,(X/Y)$.
\end{lem}

\begin{proof} First of all, since $\|x_0\|=1$, it follows from \eqref{eqn:meqextr} that the coset $x+Y$ has norm $1$ in $X/Y$. We also see from \eqref{eqn:meqextr} that $x_0$ is the unique element of minimum norm in $x+Y$, and so 
\begin{equation}\label{eqn:grone}
\|x_0+y\|>1\quad\text{\rm for all }\,y\in Y\setminus\{0\}.
\end{equation}
Now suppose that, for some $w\in X$, 
\begin{equation}\label{eqn:twonormslone}
\|x_0+w+Y\|_{X/Y}\le1\quad\text{\rm and}\quad\|x_0-w+Y\|_{X/Y}\le1.
\end{equation}
To prove that $x+Y(=x_0+Y)$ is an extreme point of $\text{\rm ball}\,(X/Y)$, we must show that $w\in Y$. 
\par Since $Y$ is proximinal, each of the two cosets in \eqref{eqn:twonormslone} contains an element of minimum norm, and we may rephrase \eqref{eqn:twonormslone} by saying that there exist vectors $y_1,y_2\in Y$ such that
\begin{equation}\label{eqn:twonormslonebis}
\|x_0+w+y_1\|\le1\quad\text{\rm and}\quad\|x_0-w+y_2\|\le1.
\end{equation}
In view of the identity 
$$x_0+\f12\left(y_1+y_2\right)=\f12\left(x_0+w+y_1\right)+\f12\left(x_0-w+y_2\right),$$
inequalities \eqref{eqn:twonormslonebis} imply that 
$$\left\|x_0+\f12\left(y_1+y_2\right)\right\|\le1.$$
When coupled with \eqref{eqn:grone}, this tells us that $y_1+y_2=0$. We may therefore rewrite \eqref{eqn:twonormslonebis} in the form 
$$\|x_0\pm(w+y_1)\|\le1,$$
and since $x_0$ is an extreme point of $\text{\rm ball}\,(X)$, we finally conclude that $w+y_1=0$. Thus $w\in Y$, as required.
\end{proof}

\section{Proof of Theorem \ref{thm:hone}} 

That (ii.1) implies (i.1) is a straightforward consequence of Lemma \ref{lem:lemtwo}, applied with $X=H^1$ and $Y=E$. The hypotheses of that lemma are, in fact, fulfilled by virtue of Lemma \ref{lem:lemone}. One should also bear in mind that, by the de Leeuw--Rudin theorem, the outer functions of norm $1$ are precisely the extreme points of $\text{\rm ball}\,(H^1)$.

\par Now let us prove that (i.1) implies (ii.1). Since $E$ is proximinal in $H^1$, as Lemma \ref{lem:lemone} tells us, we know that the set $\mathcal M\left(f;H^1,E\right)$ is nonempty. Consequently, for (ii.1) to fail, this set must contain either a non-outer function or two distinct functions (always of norm $1$). Our plan is therefore to show that each of these situations is incompatible with (i.1).

\par Suppose that $f_0$ is a function from $f+E$ satisfying $\|f_0\|_1=1$ and $f_0=Fu$, where $F$ is outer and $u$ is inner; assume in addition that $u$ is nonconstant. Replacing $u$ by $\la u$ if necessary, where $\la$ is a suitable complex number of modulus $1$, we may ensure that 
\begin{equation}\label{eqn:inteqzero}
\int_\T|f_0|\cdot\text{\rm Re}\,u\,dm=0.
\end{equation}
To see why, just note that the real-valued function 
$$\varphi(\la):=\int_\T|f_0|\cdot\text{\rm Re}(\la u)\,dm,\qquad\la\in\T,$$
is continuous on the circle and satisfies $\varphi(-1)=-\varphi(1)$. 
\par Now that \eqref{eqn:inteqzero} is achieved, we put $g:=f_0\cdot\text{\rm Re}\,u$. Since $u$ is nonconstant, $g$ is non-null. Moreover, since 
$$g=\f12Fu(\ov u+u)=\f12F\left(1+u^2\right),$$
it follows that $g$ is an outer function in $H^1$. Recalling \eqref{eqn:inteqzero}, we also see that
\begin{equation}\label{eqn:fnoughtpmg}
\|f_0\pm g\|_1=\int_\T|f_0|\,(1\pm\text{\rm Re}\,u)\,dm=\int_\T|f_0|\,dm=1.
\end{equation}
Hence 
\begin{equation}\label{eqn:twonormshonee}
\|f_0+g+E\|_{H^1/E}\le1\quad\text{\rm and}\quad\|f_0-g+E\|_{H^1/E}\le1.
\end{equation}
We also note that the two cosets involved, namely $f_0+g+E$ and $f_0-g+E$, are different because $g\notin E$. Indeed, while $g$ is outer, the functions in $E$ are all divisible by the inner factor $\th$; and if the latter happens to be constant, then every function from $E(=H^1_\Phi)$ vanishes at the origin, since $\Phi$ contains the constant function $1$. Consequently, \eqref{eqn:twonormshonee} implies that the coset $f_0+E(=f+E)$ is a nonextreme point of $\text{\rm ball}\,(H^1/E)$, contradicting (i.1).

\par We have just seen that $\mathcal M\left(f;H^1,E\right)$ cannot contain a non-outer function, as long as (i.1) is fulfilled. Now assume that
\begin{equation}\label{eqn:cardmgetwo}
\#\mathcal M\left(f;H^1,E\right)\ge2,
\end{equation}
so that $\mathcal M\left(f;H^1,E\right)$ contains two distinct functions, say $f_1$ and $f_2$. Their midpoint $\f12(f_1+f_2)$ will then also belong to $\mathcal M\left(f;H^1,E\right)$. Moreover, since this midpoint is obviously a (unit-norm) nonextreme point of $\text{\rm ball}\,(H^1)$, it will be a non-outer function in $\mathcal M\left(f;H^1,E\right)$. Consequently, the assumption \eqref{eqn:cardmgetwo} contradicts (i.1) as well. The proof is complete. 

\section{Proof of Theorem \ref{thm:kertoe}}

The implication (ii.2)$\implies$(i.2) follows immediately upon applying Lemma \ref{lem:lemtwo}, coupled with the fact that the extreme points of $\text{\rm ball}\,(K_1(\ph))$ are characterized (among the unit-norm functions) by condition \eqref{eqn:innrelprime}. 

\par To prove the converse, we begin by showing that if (i.2) holds, then the (nonempty) set $\mathcal M\left(f;K_1(\ph),E_{\ph,\psi}\right)$ cannot contain any nonextreme point of $\text{\rm ball}\,(K_1(\ph))$. Assume to the contrary that $\mathcal M\left(f;K_1(\ph),E_{\ph,\psi}\right)$ does contain such a function, say $f_0$. This means that $f_0$ is a unit-norm element of $f+E_{\ph,\psi}$ and that the inner factors of $f_0$ and $\ov{z\ph f_0}$ are both divisible by a certain nonconstant inner function, say $u$. In particular, $f_0$ is writable in the form 
\begin{equation}\label{eqn:bulba}
f_0=FI_0u, 
\end{equation}
where $F$ is outer and $I_0$ is inner. Arguing as in the proof of Theorem \ref{thm:hone}, we may further assume that our current $f_0$ and $u$ satisfy \eqref{eqn:inteqzero}. 

\par We now put $g:=f_0\cdot\text{\rm Re}\,u$, and we claim that
\begin{equation}\label{eqn:ginknotine}
g\in K_1(\ph)\setminus E_{\ph,\psi}.
\end{equation}
To verify that $g\in K_1(\ph)$, we need to check that $\ov{z\ph g}\in H^1$. This follows from the identity
$$\ov{z\ph g}=\ov{z\ph f_0}\cdot\text{\rm Re}\,u=\f12\ov{z\ph f_0}\,(u+\ov u)
=\f12\ov{z\ph f_0}\,(u+1/u),$$
since $\ov{z\ph f_0}$ is an $H^1$ function whose inner factor is divisible by $u$. Furthermore, recalling \eqref{eqn:bulba}, we get 
$$g=\f12FI_0u\,(\ov u+u)=\f12F(1+u^2)I_0,$$
which shows that the inner factor of $g$ is $I_0$. Finally, we observe that $I_0$ is not divisible by $\psi$; otherwise, \eqref{eqn:bulba} would imply that $f_0\in\psi H^1$ (and hence $f_0\in E_{\ph,\psi}$), so the coset $f_0+E_{\ph,\psi}\left(=f+E_{\ph,\psi}\right)$ would be null, which is not the case. We readily conclude that $g\notin\psi H^1$, and \eqref{eqn:ginknotine} is thereby established.

\par From \eqref{eqn:ginknotine} we infer that 
$$f_0+g+E_{\ph,\psi}\ne f_0-g+E_{\ph,\psi}.$$
Also, since $\|f_0\pm g\|_1=1$ by virtue of \eqref{eqn:fnoughtpmg}, we have 
$$\|f_0+g+E_{\ph,\psi}\|\le1\quad\text{\rm and}
\quad\|f_0-g+E_{\ph,\psi}\|\le1,$$
the norm $\|\cdot\|$ being that of the quotient space $K_1(\ph)/E_{\ph,\psi}$. This, however, tells us that the coset $f_0+E_{\ph,\psi}\left(=f+E_{\ph,\psi}\right)$ is a non-extreme point of $\text{\rm ball}\,\left(K_1(\ph)/E_{\ph,\psi}\right)$, contradicting (i.2). We have thus proved that the set $\mathcal M\left(f;K_1(\ph),E_{\ph,\psi}\right)$ does not contain any non-extreme point of $\text{\rm ball}\,(K_1(\ph))$, provided that (i.2) is fulfilled. 

\par To discard the possibility that 
\begin{equation}\label{eqn:cardgetwotoep}
\#\mathcal M\left(f;K_1(\ph),E_{\ph,\psi}\right)\ge2,
\end{equation}
we argue---just as we did in the preceding proof---that if $f_1$ and $f_2$ were two distinct functions in $\mathcal M\left(f;K_1(\ph),E_{\ph,\psi}\right)$, then their midpoint $\f12(f_1+f_2)$ would also belong to $\mathcal M\left(f;K_1(\ph),E_{\ph,\psi}\right)$. At the same time, this midpoint would obviously be a nonextreme point of $\text{\rm ball}\,(K_1(\ph))$, so we infer that \eqref{eqn:cardgetwotoep} would indeed contradict (i.2). 

\par To summarize, (i.2) implies that $\mathcal M\left(f;K_1(\ph),E_{\ph,\psi}\right)$ contains exactly one function, say $h$, which is an extreme point of $\text{\rm ball}\,(K_1(\ph))$. This last condition can be rephrased by saying that the inner factors of $h$ and $\ov{z\ph h}$ are relatively prime, and we arrive at (ii.2). The proof is complete.

\section{Proof of Theorem \ref{thm:hinfty}}

The strategy to be employed is similar to that of the preceding proofs. One of the two implications, namely (ii.3)$\implies$(i.3), is again verified immediately by applying Lemma \ref{lem:lemtwo}, this time with $X=H^\infty$ and $Y=E$. The lemma is indeed applicable, since any weak-star closed subspace of a dual space is proximinal. 

\par As a first step in proving that (i.3) implies (ii.3), we shall deduce from the former condition that the set $\mathcal M\left(f;H^\infty,E\right)$ does not contain any nonextreme point of $\text{\rm ball}\,(H^\infty)$. Assuming the contrary, we could find a unit-norm function $f_0\in H^\infty$ satisfying $f_0-f\in E$ and $\log(1-|f_0|)\in L^1$. Let $F$ be the outer function with modulus $1-|f_0|$, so that 
$$F(z):=\exp\left\{\int_\T\f{\ze+z}{\ze-z}\log(1-|f_0(\ze)|)\,dm(\ze)\right\},\qquad z\in\D.$$
We further claim that $FH^\infty\not\subset E$, or in other words, that there exists a function $G\in H^\infty$ with
\begin{equation}\label{eqn:fgnotine}
FG\notin E.
\end{equation}
Indeed, if our (weak-star closed) subspace $E$ contained the set $FH^\infty$, it would also contain the weak-star closure of $FH^\infty$, which actually coincides with $H^\infty$. (Since $F$ is outer, it follows that $FH^\infty$ is weak-star dense in $H^\infty$; see \cite[pp.\,81--82]{G}.) We would then have $E=H^\infty$, which is not the case by hypothesis.

\par The existence of a function $G\in H^\infty$ satisfying \eqref{eqn:fgnotine} is thus established. Such a $G$ is obviously non-null, and we may normalize it to get $\|G\|_\infty=1$. This done, we put $g:=FG$. Since
$$|f_0\pm g|\le|f_0|+|g|\le|f_0|+|F|=1$$
almost everywhere on $\T$, we see that $\|f\pm g\|_\infty\le1$. This in turn implies that 
\begin{equation}\label{eqn:twonormshinftye}
\|f_0+g+E\|_{H^\infty/E}\le1\quad\text{\rm and}\quad\|f_0-g+E\|_{H^\infty/E}\le1.
\end{equation}
Moreover, since $g(=FG)$ is not in $E$, we have $f_0+g+E\ne f_0-g+E$. Therefore, inequalities \eqref{eqn:twonormshinftye} tell us that the coset $f_0+E(=f+E)$ is a nonextreme point of $\text{\rm ball}\,(H^\infty/E)$, contradicting (i.3).

\par Thus, (i.3) implies that no nonextreme point of $\text{\rm ball}\,(H^\infty)$ can be in $\mathcal M\left(f;H^\infty,E\right)$. To arrive at (ii.3), it now suffices to show that the (nonempty) set $\mathcal M\left(f;H^\infty,E\right)$ contains exactly one element. In fact, if $f_1$ and $f_2$ were two distinct functions lying in that set, then their midpoint $\f12(f_1+f_2)=:f_*$ would also belong to it; besides, $f_*$ would then be a nonextreme point of $\text{\rm ball}\,(H^\infty)$. Such a situation is incompatible with (i.3), as we know, so we eventually conclude that 
$$\#\mathcal M\left(f;H^\infty,E\right)=1,$$
and we are done.

\end{document}